\RequirePackage{fix-cm}
\documentclass[smallcondensed,natbib]{svjour3}     
\smartqed  
\usepackage{graphicx}

\usepackage[all]{xy} 
\usepackage{amsmath,amssymb,amstext} 
\usepackage{dsfont} 
\usepackage{mathrsfs}
\usepackage{url}
\usepackage{natbib}

\newcommand{\MZ}{\mathbb{Z}}
\newcommand{\MQ}{\mathbb{Q}}
\newcommand{\MN}{\mathbb{N}}
\newcommand{\MR}{\mathbb{R}}
\newcommand{\MC}{\mathbb{C}}
\newcommand{\MH}{\mathbb{H}}
\renewcommand{\H}{\mathrm{H}}

\renewcommand{\O}{\mathcal{O}}
\renewcommand{\c}{\mathfrak{c}}
\newcommand{\End}{\mathrm{End}}
\newcommand{\trace}{\mathrm{Tr}}
\renewcommand{\P}{\mathcal{P}}
\newcommand{\GL}{\mathrm{GL}}
\renewcommand{\S}{\mathcal{S}}
\newcommand{\stab}{\mathrm{Stab}}
\renewcommand{\a}{\mathfrak{a}}
\newcommand{\Cl}{\mathrm{Cl}}
\newcommand{\WR}{\Sigma^{\mathrm{wr}}}
\newcommand{\WRF}{\WR_{=1}}
\newcommand{\Aff}{\mathrm{Aff}}
\newcommand{\SL}{\mathrm{SL}}
\newcommand{\PSL}{\mathrm{PSL}}
\newcommand{\fp}{\mathfrak{p}}
\begin{document}

\title{Resolutions for unit groups of orders}

\author{Sebastian Sch\"onnenbeck
}

\institute{ Sebastian Sch\"onnenbeck \at
              Lehrstuhl D f\"ur Mathematik, RWTH Aachen \\
	      Pontdriesch 14/16, D-52062 Aachen, Germany\\
              \email{sebastian.schoennenbeck@rwth-aachen.de}\\
	    \url{www.math.rwth-aachen.de/homes/Sebastian.Schoennenbeck}%
}
\date{}

\maketitle

\begin{abstract}
We present a general algorithm for constructing a free resolution for unit groups of orders in semisimple rational algebras. The approach is based on computing a contractible $G$-complex employing the theory of minimal classes of quadratic forms and Opgenorth's theory of dual cones. The information from the complex is then used together with Wall's perturbation lemma to obtain the resolution.

\end{abstract}

\section{Introduction}\label{intro}
Let $K$ be an imaginary quadratic number field with ring of integers $\MZ_K$. The integral homology of $\GL_2(\MZ_K)$ (or often rather the closely related group $\PSL_2(\MZ_K)$ known as a Bianchi group) is a widely studied concept in the literature (see for example \cite{BerkoveCohomology, rahm, SchwermerVogtmann, Cremona}). In higher dimensions considerably less is known but there are still some results available (see \cite{PSL4, GunnelsCohomologyOfLinearGroups}). 

The group $\GL_2(\MZ_K)$ can be thought of as the unit group of the maximal $\MZ_K$-order $\End_{\MZ_K}(\MZ_K^2)$ in the simple $\MQ$-algebra $K^{2\times 2}$. Thus questions about the homology of $\GL_2(\MZ_K)$ naturally generalize to questions about the homology of unit groups of general maximal orders in simple $\MQ$-algebras. In this article we present an algorithm to construct a free resolution in this very general setup. The approach generalizes the ideas used in \cite{PSL4} to compute the integral homology of $\PSL_4(\MZ)$ and \cite{GunnelsCohomologyOfLinearGroups} to compute the integral homology of $\GL_3$ over imaginary quadratic integers. Moreover the ideas presented here were recently used to compute maximal subgroups and presentations of these unit groups (see \cite{MaxMin,Braunetal}). We start out by constructing a cell complex with suitable action in the Voronoi cone of positive definite forms (following \cite{Opgenorth} which itself is a generalization of \cite{KoecherReduktionstheorie}). The resulting cell complex does not directly yield a free resolution (seeing as there are usually nontrivial cell stabilizers), however, it can be used in a subsequent step together with Wall's perturbation lemma (\cite{Wall}) to actually construct one.

The algorithm has been implemented (using a combination of \cite{Magma, GAP4,HAP}) in a few sample cases. In particular we were able to construct cell complexes for certain groups of type $\SL_2$ over imaginary quadratic integers having only finite cell stabilizers (where previously only complexes with infinite stabilizers were known) and we supplement some of the results of \cite{GunnelsCohomologyOfLinearGroups} by computing the torsion in the integral homology of $\GL_3(\MZ_K)$ for $K=\MQ(\sqrt{-7})$ and $K=\MQ(\sqrt{-1})$ in low dimensions. 
 
The article is organized as follows. We start by reviewing some basic notions on maximal orders in simple $\MQ$-algebras and the associated cone of positive definite forms. In section \ref{Complexsection} we then construct the cell complex of minimal classes and study its geometry. Section \ref{Perturbationsection} is devoted to the construction of a free resolution starting from the non-free resolution we obtain from the cell complex and finally in Section \ref{Computationalsection} we present some examples of the results we were able to achieve using this method.
\section{Prerequisites}\label{Presection}
In this section we want to review some facts on maximal orders in simple $\MQ$-algebras and the cone of positive definite forms. A standard reference for the former is \cite{Reiner} and for the latter we refer to \cite{MaxMin} and \cite{Braunetal}.
\subsection{Maximal orders}
Let us first fix some notation. For the remainder of this article let $A$ denote a finite-dimensional simple 
$\MQ$-algebra. Then we have $A \cong D^{n \times n}$ for some finite-dimensional $\MQ$-division 
algebra $D$ and $n \in \MN$. Furthermore we will write $K=\mathrm{Z}(D)$ for the center of $K$, 
$\MZ_K:=\mathrm{Int}_\MZ(K)$ for its ring of integers, and denote a maximal $\MZ_K$-order in D by $\O$.

\begin{definition}
 An $\O$-lattice of rank $n$ is a finitely generated $\O$-submodule of the right $D$-module $V:=D^n$ containing 
a $D$-basis of $V$.
\end{definition}

\begin{remark}
 \begin{enumerate}
  \item Let $L$ be some $\O$-lattice of rank $n$. Steinitz's theorem \cite[Thm 4.13, Cor. 
35.11]{Reiner} implies that there exist right $\O$-ideals $\c_1,...,\c_n$ as well as a $D$-basis $e_1,...,e_n$ 
of $V$ such that $L=e_1\c_1 \oplus ... \oplus e_n\c_n$. The Steinitz invariant of $L$ is defined to be the 
class 
$\mathrm{St}(L):=\left[\c_1\right] +...+\left[c_n \right]$ in the group of stable isomorphism classes of right 
$\O$-ideals.
\item If $n\geq 2$ two lattices $L_1,L_2$ of equal rank are isomorphic if and only if 
$\mathrm{St}(L_1)=\mathrm{St}(L_2)$. In particular if $\mathrm{St}(L_1)=\left[\c \right]$ 
we have $L_1 \cong \O^{n-1} \oplus \c$.
\item The endomorphism ring $\End_\O(L)=\{X \in A~|~XL=L\}$ is a maximal order in $\End_D(V)\cong A$ and any 
maximal order in $A$ is of this form \cite[Cor. 27.6]{Reiner}.
 \end{enumerate}
\end{remark}

\subsection{The cone of positive definite forms}
Taking the scalar extension of our simple $\MQ$-algebra $A$ with the reals yields a semisimple 
$\MR$-algebra 
$A_\MR:=A \otimes_\MQ \MR$ which is therefore isomorphic to a direct sum of matrix rings over $\MR$, $\MC$ and 
$\MH$ (the quaternions), respectively. We resort to the notation of \cite{MaxMin} and set $d^2:=\dim_K(D)$. 
Moreover let $s$ be the number of real places of $K$ which ramify in $D$, $r$ the number of real places of 
$K$ which do not ramify in $D$, and $t$ the number of complex places of $K$. Having fixed this notation we 
know that
\begin{equation}
 D_\MR:=D\otimes_\MQ \MR \cong \bigoplus_{i=1}^s \MH^{d/2 \times d/2} \oplus \bigoplus_{i=1}^r \MR^{d \times d} 
\oplus \bigoplus_{i=1}^t \MC^{d \times d}.
\end{equation}
On this algebra we define the involution $^*$ (which is canonical up to the choice of the above isomorphism) 
componentwise to be transposition on the matrix rings over $\MR$ and transposition and entrywise conjugation 
(complex or quaternionic, respectively) on the matrix rings over $\MC$ or $\MH$. This definition yields in the 
usual way a map 
\begin{equation}
 ^\dagger: D_\MR ^{m \times n} \rightarrow D_\MR^{n \times m}
\end{equation}
via transposition and applying $^*$ componentwise, in particular we get an involution on $A_\MR \cong D_\MR^{n 
\times 
n}$.

\begin{definition}
 \begin{enumerate}
  \item Let $\Sigma:=\{F \in A_\MR ~|~F^\dagger=F \}\subset A_\MR$ denote the $\MR$-subspace of 
$\dagger$-Hermitian elements of $A_\MR$.
\item On $\Sigma$ we may define a positive definite bilinear form via
\begin{equation}
 \langle F_1,F_2 \rangle:=\trace(F_1F_2)
\end{equation}
where $\trace$ indicates the reduced trace of the semisimple $\MR$-algebra $A_\MR$.
\item $\P$ will denote the cone of positive definite elements in $\Sigma$:
\begin{equation}
 \P:=\{(q_1,...,q_s,f_1,..,f_r,h_1,...,h_t) \in \Sigma~|~ q_i,f_j,h_k \text{ positive definite} \}.
\end{equation}
 \end{enumerate}
\end{definition}
The elements of $V_\MR:=V \otimes_\MQ \MR$ correspond to elements of $\Sigma$ in the sense that for
each $x \in V_\MR:=V \otimes_\MQ \MR$ we have $xx^\dagger \in \Sigma$.
\begin{lemma}{\cite[Lemma 3.2]{MaxMin}}
 Let $F \in \Sigma$, then $F$ defines a quadratic form on $V_\MR$ via
 \begin{equation}
  F\left[x\right]:=\langle F,xx^\dagger \rangle, ~x \in V_\MR.
 \end{equation}
This form is positive definite if and only if $F \in \P$.
\end{lemma}
With this lemma in mind we shall also refer to the elements of $\Sigma$ as 'forms'.

\begin{lemma}
 \begin{enumerate}
  \item $\P$ is an open subset of $\Sigma$.
  \item Let $F_1,F_2 \in \P$ then $\langle F_1,F_2 \rangle >0$.
  \item For all $F_1 \in \Sigma-\P$ there exists $0 \neq F_2 \in \overline{\P}$ with $\langle F_1,F_2 \rangle 
\leq 0$, where $\overline{\P}$ denotes the topological closure of $\P$. 
 \end{enumerate}
\end{lemma}
\begin{proof} \begin{enumerate}
                 \item This is well known.
                 \item We may prove this componentwise. But then the spectral theorem holds and we may assume 
$F_1$ to be a diagonal matrix with positive real entries. In this case the assertion is obvious.
		\item Let $F \in \Sigma-\P$. In this case the form which $F$ defines on $V_\MR$ is not positive 
and there exists some $0\neq x \in V_\MR$ with $\langle F,xx^\dagger \rangle \leq 0$. Now $xx^\dagger$ is 
positive semidefinite and therefore in the topological closure of $\P$.\qed
                \end{enumerate}
 \end{proof}

The preceding lemma shows that $\P \subset \Sigma$ constitutes a self-dual cone in the sense of 
\cite{Opgenorth}.
\section{The CW-complex of well-rounded forms}\label{Complexsection}
We keep the notation from the previous section and are now prepared to construct (for a given unit group of a maximal order) a cell complex with a suitable 
action. The cell complex is constructed by decomposing the cone $\P$ into so called minimal classes which were first introduced in \cite{ash1984} which is also our primary source.
\subsection{Minimal classes}
We will first introduce the notion of minimal classes which will define the desired cell structure. Most of the 
paragraph will closely follow the ideas of \cite{MaxMin}.
For the remainder of this section let $L$ denote a fixed $\O$-lattice of rank $n$, $\Lambda:=\End_\O(L)$ its 
ring of endomorphisms, and $\GL(L):=\Lambda^*$ the corresponding unit group. 
\begin{lemma}\label{admissiable}
 The set $M_L:=\{ll^\dagger~|~ 0 \neq l \in L \}$ is discrete in $A_\MR$ and admissible in the sense of 
\cite{Opgenorth}, i.e. for each sequence $(F_i)_{i \geq 1} \subset \P$ converging to an element $F \in 
\overline{\P}-\P$ the sequence $(\min_{x \in M_L}F_i\left[x\right])_{i \geq 1}$ converges to $0$. 
\end{lemma}
\begin{proof} $L$ is discrete in $V_\MR$, hence $M_L$ is discrete in $\Sigma$. Let now $F \in 
\overline{\P}-\P$ be a positive semidefinite form. It suffices to construct a sequence $(l_i)_i\subset L-\{0\}$ 
such that $(F\left[l_i \right])_i$ converges to $0$. To do this we decompose $V_\MR=\mathrm{rad}(F) \oplus U$ 
into the radical of the form defined by $F$ and an arbitrary complement $U$ (in particular $F_{|U}$ is positive 
definite). Now let $U_\epsilon$ be the open ball of radius $\epsilon$ around $0$ in $U$ relative to the norm defined on 
$U$ by $F$. Then $\mathrm{rad}(F) \oplus U_\epsilon$ is convex, centrally symmetric around the origin and of 
infinite volume. Minkowski's lattice point theorem then implies that there is some $0 \neq l_\epsilon \in L 
\cap (\mathrm{rad}(F) \oplus U_\epsilon)$ and we have $F\left[l_\epsilon \right] \leq \epsilon$. We may 
therefore construct the desired sequence and the assertion follows. \qed \end{proof}

We will now give a short introduction to the notion of weights which was introduced in 
\cite{ash1984} and which allows us to construct several non-equivalent cell decompositions of the same space.

\begin{definition}
 \begin{enumerate}
  \item A weight $\varphi$ on $L$ is a $\GL(L)$-invariant map from the projective space $\mathbf{P}(D^n)$ to 
the 
positive reals with maximum $1$.
\item By $\varphi_0$ we will denote the trivial weight, i.e. $\varphi_0(x):=1$ for all $x$.
 \end{enumerate}
\end{definition}
In what follows we will not strictly distinguish between a weight $\varphi: \mathbf{P}(D^n) \rightarrow 
\MR_{>0}$ and the induced map $L-\{0\} \rightarrow \MR_{>0}, l \mapsto \varphi(lD)$.

\begin{definition}\begin{enumerate}
                   \item Let $L=e_1 \c_1 \oplus...\oplus e_n\c_n$ be a lattice. To $0 \neq l=\sum_{i=1}^n e_il_i 
\in L$ we associate the integral left $\O$-ideal $\a_l:=\sum_{i=1}^n \c_i^{-1}l_i$ and its integral norm 
$\mathrm{N}(\a_l):=|\O/\a_l|=\mathrm{N}_{K/\MQ}(\mathrm{nr}(\a_l)^d)$.
\item Let $x \in D^n$ and $0\neq \lambda \in D$ arbitrary with $x\lambda \in L$. We define
\begin{equation}
 \mathrm{N}_x:=\mathrm{N}(\left[ \mathrm{nr}(\a_{x\lambda}) \right])=\min_{I \triangleleft \O, 
\left[\mathrm{nr} (I) \right]=\left[\mathrm{nr}(\a_{x\lambda})\right]}\mathrm{N}_{K/\MQ}(\mathrm{nr}(I)^d).
\end{equation}
\item For $x \in D^n$ we set $\varphi_1(x):=\mathrm{N}_x^{-2/|K:\MQ|}$.
\end{enumerate}
\end{definition}
\begin{proposition}{\cite[Lemma 4.3]{MaxMin}}
 The function \begin{equation}\varphi_1:\mathbf{P}(D^n)\rightarrow \MR,~\left[x\right] \mapsto \varphi_1(x)\end{equation}is a well-defined 
weight.
\end{proposition}

We mention the strange looking weight $\varphi_1$ at this point as in some cases (e.g. when $D$ is an imaginary 
quadratic number field) it is in some sense more natural to work with $\varphi_1$ instead of 
$\varphi_0$ (see \cite{CoulangeonHabil} for an exposition). For the remainder of this section we will now fix - 
in addition to the lattice $L$ - a weight 
$\varphi$.

\begin{definition}
 \begin{enumerate}
  \item Given $F \in \P$ we define the $L$-minimum of $F$ with respect to $\varphi$ to be
  \begin{equation}
   \mathrm{min}_L(F):=\min_{0 \neq l \in L} \varphi(l)F\left[l\right].
  \end{equation}
\item The set of shortest vectors in $L$ with respect to $F$ and $\varphi$ is then denoted by
\begin{equation}
 \S_L(F):=\{ l \in L ~|~ \varphi(l)F\left[l\right]=\mathrm{min}_L(F) \}.
\end{equation}
 \end{enumerate}
\end{definition}
The following remark is essential for all computations with these definitions.
\begin{remark}
 Let $F \in \P$. The set $\S_L(F)$ is finite, since
 \begin{equation}
  \S_L(F) \subset \left\{0 \neq l \in L~|~ F\left[l\right] \leq \frac{\mathrm{min}_L(F)}{\min_{0 \neq y \in L} 
\varphi(y)}\right\}
 \end{equation}
which is a set of shortest vectors in a $\MZ$-lattice and hence finite. The set is well-defined because 
$\varphi$ 
attains only finitely many values.
\end{remark}
We are now prepared to define the cell decomposition of $\P$ with respect to the lattice $L$ and the weight 
$\varphi$.
\begin{definition}
 \begin{enumerate}
  \item Let $F \in \P$. $\Cl_L(F):=\{ H \in \P~|~ \S_L(H)=\S_L(F) \}$ is called the minimal class of $F$.
  \item If $C=\Cl_L(F)$ is a minimal class we set $\S_L(C):=\S_L(F)$.
  \item A minimal class $C$ is called well-rounded if $\S_L(C)$ contains a $D$-basis of $V$.
  \item $F \in \P$ is called perfect if $\Cl_L(F)=\{aF~|~a \in \MR_{>0}\}$.
 \end{enumerate}
\end{definition}
Note that all these definitions depend on the choice of the lattice $L$ as well as the choice of the 
weight $\varphi$. 

\begin{remark}
 $\GL(L)$ acts on $\P$ via $gF:=gFg^\dagger$ and this action respects the decomposition of $\P$ into the 
minimal classes. We therefore get a natural action of $\GL(L)$ on the set of minimal classes.
\end{remark}
\begin{lemma}[\protect{\cite[Lemma 5.3]{MaxMin}}]
 Let $C \subset \P$ be a well-rounded minimal class and set $T_C:=\sum_{x \in \S_L(C)}xx^\dagger$. Then $T_C$ 
is a positive definite form, $\stab_{\GL(L)}(C)=\stab_{\GL(L)}(T_C^{-1})$, and the class $C'$ is in the same 
$\GL(L)$-orbit as $C$ if and only if $T_C^{-1}$ and $T_{C'}^{-1}$ lie in the same $\GL(L)$-orbit.
\end{lemma}
Thus it is computationally easy to construct stabilizers of minimal classes or check whether two minimal classes lie in the same $\GL(L)$ orbit.
\begin{remark}
 The above lemma directly implies that $\stab_{\GL(L)}(C)$ is finite for any well-rounded class $C$.
\end{remark}

\subsection{A closer look at the decomposition}

Let us now take a closer look at our decomposition and its geometry. 

First of all note that $\P$ is a convex and thus contractible topological space. However we do not want to use $\P$ for our computations since we have already seen that we can only guarantee finiteness of cell stabilizers for well-rounded minimal classes. Let us introduce some additional notation.
\begin{definition}
 \begin{enumerate}
  \item We will use the notation $\WR:=\{F \in \P~|~ F \text{ well-rounded}\}$ for the space of well-rounded 
forms.
\item $\WRF$ will denote the space of well-rounded forms whose $L$-minimum is $1$.
 \end{enumerate}
 Since any form may be rescaled to have minimum $1$ we will from now on think of well-rounded minimal classes as subsets of 
$\WRF$.
 \end{definition}

The following result shows us that $\WRF$ is still a suitable candidate to help us in constructing a free resolution.
\begin{proposition}[\protect{\cite[Thm. 1]{SouleCohomology}}]
 The space $\WR$ (and thus $\WRF$) is a ($\GL(L)$-invariant) deformation retract of $\P$ and thus contractible.
\end{proposition}

As we already know about some finiteness results it is important to note that the action of $\GL(L)$ on the cells in $\WRF$ 
admits only finitely many orbits. To see this we need the following result of A. Ash.
\begin{theorem}[\protect{\cite[Thm. (ii)]{ash1984}}]
 The set $\WRF/\GL(L)$ is compact.
\end{theorem}
\begin{corollary}
 Up to positive real homotheties and the action of $\GL(L)$ there are only finitely many perfect forms.
\end{corollary}
\begin{proof} Lemma \ref{admissiable} implies that \cite[Lemma 1.6]{Opgenorth} holds. Hence the set of perfect 
forms in $\WRF$ is discrete. Together with Ash's aforementioned theorem this implies the assertion. \qed \end{proof}

To show that this finiteness result also holds for the minimal classes of non-perfect forms we first need some 
further results.
\begin{lemma}
 Let $C \subset \WRF$ be a minimal class. Then $C$ is convex.
\end{lemma}
\begin{proof} Let $F_1, F_2 \in C$ and $\lambda \in \left[0,1\right]$. Then for all $l \in L$ we have
\begin{equation}
 \lambda \varphi(l) F_1\left[l\right]+(1-\lambda)\varphi(l)F_2\left[l\right] \geq \lambda \mathrm{min}_L(F_1)+ 
(1-\lambda) \mathrm{min}_L(F_2)=1. 
\end{equation}
And equality holds exactly for $l \in \S_L(F_1)=\S_L(F_2)$. Hence $C$ is convex. \qed \end{proof}

\begin{definition}
 On the set of minimal classes we define the following partial ordering:
 \begin{equation}
  C \preceq C' :\Leftrightarrow \S_L(C) \subset \S_L(C').
 \end{equation}

\end{definition}
Obviously this partial ordering is compatible with the $\GL(L)$-action. In fact the following lemma shows that it coincides with the inclusion ordering arising from the cell decomposition.
\begin{lemma}
 Let $C \subset \WRF$ be a minimal class and $\overline{C}$ its topological closure in $\WRF$. Then
 \begin{equation}
  \overline{C} = \bigcup_{C' \succeq C} C'.
 \end{equation}
\end{lemma}
\begin{proof} Let $(F_i)_i \subset C$ be a sequence converging to $F$. Since taking the $L$-minimum is 
continuous, $F$ will again have $L$-minimum $1$ and attain this value on the elements of $\S_L(C)$. Hence by 
Lemma \ref{admissiable} we have $F \in \P$ and $\S_L(C) \subset \S_L(F)$ so $F$ lies in the above union. This shows 
that $\bigcup_{C' \succeq C} C'$ is in fact a closed set.

On the other hand let $F'$ be in said union and $F 
\in C$. We set $F_\lambda:=(1-\lambda)F+\lambda F'$ and see that $F_\lambda \in C$ for all $\lambda \in 
\left[0,1\right)$ since
\begin{equation}
 \varphi(l)F_\lambda\left[l \right] =\varphi(l)(1-\lambda) F\left[l\right] +\varphi(l)\lambda F'\left[l\right] 
\geq 
(1-\lambda)+\lambda \varphi(l) F'\left[l\right] \geq 1
\end{equation}
with equality if and only if $l \in \S_L(F)=\S_L(C)$. Hence there is a sequence of elements of $C$ converging 
to $F'$ and the assertion holds. \qed \end{proof}

\begin{lemma}
 Let $C \subset \WRF$ be a minimal class. Then the topological closure of $C$ contains only finitely many 
perfect forms.
\end{lemma}
\begin{proof} Since we already know that there are only finitely many perfect forms up to the action of 
$\GL(L)$ it suffices to show that for a given perfect form $F$ there are only finitely many $g \in \GL(L)$ with 
$gFg^\dagger \in \overline{C}$. Now the following holds
\begin{equation}
 gFg^\dagger \in \overline{C} \Leftrightarrow \S_L(gFg^\dagger) \supset \S_L(C) \Leftrightarrow g^{-\dagger} 
\S_L(F) \supset \S_L(C) \Leftrightarrow \S_L(F) \supset g^\dagger \S_L(C).
\end{equation}
Now $\S_L(C)$ contains a basis for $V$ and thus $g^\dagger$ is uniquely determined by its values on $\S_L(C)$. 
Since $\S_L(C)$ and $\S_L(F)$ are finite there are only finitely many maps from $\S_L(C)$ to $\S_L(F)$ and 
therefore in particular only finitely many $g \in \GL(L)$ fulfilling the condition. \qed \end{proof}

Note that \cite[Prop. (1.6)]{Opgenorth} implies that there always is a perfect form in the closure of a 
minimal class.

\begin{definition}
Let $C\subset \WRF$ be a minimal class and $N:=\dim_\MR(\Sigma)$.
\begin{enumerate}
 \item The dimension $\dim_\MR(\langle xx^\dagger ~|~ x \in \S_L(C) \rangle)$ is called the perfection rank of 
$C$, the codimension $N-\dim_\MR(\langle xx^\dagger ~|~ x \in \S_L(C) \rangle)$ the perfection corank 
of $C$.
\item The affine space generated by $C$ will be denoted by $\Aff(C)$.
\end{enumerate}

\end{definition}

\begin{lemma}
 Let $C\subset \WRF$ be a minimal class. 
Then $C$ is open in $\Aff(C)$. 
\end{lemma}
\begin{proof} This a direct consequence of \citep[Lemma (1.3)]{Opgenorth} which is applicable because of 
Lemma \ref{admissiable}. \qed \end{proof}
The following lemma shows that the perfection corank is actually a feasible way to determine the dimension of a 
minimal class.

\begin{lemma}
 Let $C \subset \WRF$ be a minimal class. Then $\Aff(C)=\{F \in \Sigma ~|~ 
\varphi(l)F\left[l\right]=1~\forall~l \in \S_L(C)\}$. Moreover the (affine) dimension of $\Aff(C)$ is precisely 
the perfection corank of $C$.
\end{lemma}
\begin{proof} We set $M:=\{F \in \Sigma ~|~ \varphi(l)F\left[l\right]=1~\forall~l 
\in \S_L(C)\}$. It is easy to see that $\varphi(l)\langle F, ll^\dagger \rangle=1$ is an affine condition 
on $F$ and therefore $\dim_\MR(M)$ is in fact the perfection corank of $C$. Obviously $C \subset M$ and 
therefore $\Aff(C) \subset M$ holds. To see the converse let $F \in C$ and let $m_1,...,m_r$ be a basis for the 
space of translations of $M$. We may rescale $m_1,...,m_r$ such that $F+m_i \in \P$ for all $i$. Clearly 
$(F+m_i)\left[l\right]=1$ for all $l \in \S_L(C)$ and after rescaling the $m_i$ again we may also assume that 
$(F+m_i)\left[l\right] >1$ for all $0 \neq l \in L-\S_L(C)$. But then $F+m_i \in C$ for all $1\leq i \leq r$ 
and $C$ contains an affine basis of $M$. \qed \end{proof}

The following corollary is now easy to see.
\begin{corollary}
 The perfection rank is a strictly increasing function on the set of minimal classes with the 
partial ordering defined above.
\end{corollary}

\begin{remark}
 Let $C \subset \WRF$ be a minimal class. Then $C$ is bounded (as a subset of $\Sigma$).
\end{remark}

The last two assertions have prepared us to prove the most important of the structural properties of our cell 
decomposition.
\begin{theorem}
 Let $C \subset \WRF$ be a minimal class. Then its closure is the convex hull of the perfect forms it contains.
\end{theorem}
\begin{proof} We will prove this via induction on the perfection corank. If $C$ is the class corresponding to 
a perfect form there is nothing to show. If $C$ is not the class of a perfect form let $F \in C$ and let $F' 
\in \overline{C}$ be a perfect form. Now we know that $C$ is bounded, hence there is some $\rho>0$ such that 
$F'':=F+\rho(F-F')$ is in the boundary of $C$. Since $C$ was open in $\Aff(C)$ we have $\Cl_L(F'') \succneqq C$. 
By induction $F''$ is a convex combination of perfect forms in the closure of $C$. Now $F$ is a convex 
combination of $F'$ and $F''$ which implies the assertion. \qed \end{proof}

This theorem has some computationally very useful consequences.
\begin{corollary}
 Up to the action of $\GL(L)$ there are only finitely many well-rounded minimal classes.
\end{corollary}
\begin{corollary}
 If $C\subset \WRF$ is a minimal class there are only finitely many minimal classes contained in $\overline{C}$.
\end{corollary}
\begin{corollary}
 Let $C \subset \WRF$ be a minimal class. $\overline{C}$ is a bounded polytope whose faces are exactly the 
minimal classes $C' \succneqq C$.
\end{corollary}

\section{Homology computations}\label{Perturbationsection}
The previous section provided us with a finite-dimensional CW-complex together with a cellular $\GL(L)$-action. Hence the 
cellular chain complex arising from the decomposition of $\WRF$ into minimal classes is in fact a chain complex 
of $\GL(L)$-modules. We want to use this information to compute a free $\MZ [\GL(L)]$-resolution of $\MZ$ which 
may then be used for homology computations.

\subsection{Perturbations}
First note that the $\GL(L)$-modules appearing in the cellular chain complex of $\WRF$ are not free as each cell 
has a non-trivial stabilizer. Hence the cellular chain complex itself does not constitute a $\MZ [\GL(L)]$-free 
resolution of $\MZ$. However the following theorem originally due to C. T. C. Wall (\cite{Wall}) allows us to combine the 
cellular chain complex with free resolutions of $\MZ$ over the group rings of stabilizers of minimal classes to 
obtain the desired resolution. We do not worry here about finding resolutions for the stabilizers as all these groups are finite in which case there are computational methods readily available (see for example \cite{ellis2004computing}).

\begin{theorem}[\protect{\cite[Prop. 1, Prop. 4]{polytopal}}]
 Let $\{A_{p,q}~|~p,q \geq 0 \}$ be a bigraded family of $\MZ G$-free modules und $d_0: A_{p,q} \rightarrow 
A_{p,q-1}$ homomorphisms such that $(A_{p,*},d_0)$ is an acyclic chain complex for each $p$. We set 
$C_p:=\H_0(A_{p,*})$ and assume furthermore that there are homomorphisms $\partial:C_p \rightarrow C_{p-1}$ 
such that $(C_*,\partial)$ is a chain complex. Then the following holds:
\begin{enumerate}
 \item There are homomorphisms $d_k:A_{p,q} \rightarrow A_{p-k,q+k-1}$ for $k\geq 1,p>k$ such that
 \begin{equation}
 d=d_0+d_1+d_2+...: R_n:=\bigoplus_{p+q=n} A_{p,q}\rightarrow R_{n-1}=\bigoplus_{p+q=n-1} A_{p,q}
\end{equation}
is the differential of a chain complex $R_*$ of free $\MZ G$-modules.

\item The canonical chain map $\phi_p:A_{p,*} \rightarrow \H_0(A_{p,*})$ yields a chain map $\phi_*:R_* 
\rightarrow C_*$ which induces an isomorphism in homology.

\item Assume that there are $\MZ$-homomorphisms $h_0: A_{p,q} \rightarrow A_{p,q+1}$ with $d_0h_0d_0(x)=d_0(x)$ 
for all $x \in A_{p,q+1}$ (a so called contracting homotopy). Then we can construct $d_k$ by first lifting 
$\partial$ to $d_1: A_{p,0} \rightarrow A_{p-1,0}$  and setting \begin{equation} d_k=-h_0(\sum_{i=1}^k 
d_id_{k-i})\end{equation} recursively 
on the free generators of $A_{p,q}$.
\end{enumerate}
\end{theorem}

 In our situation this reads as follows: 
 
 Let $X:=\WRF$ denote the CW-complex of well-rounded forms of minimum $1$. For $p\in \MZ_{\geq 0}$ let $e_p$ be 
a system of representatives of minimal classes of perfection corank (dimension) $p$. 
Then
\begin{equation}
 C_p(X) \cong \bigoplus_{c \in e_p} \MZ [\GL(L)] \otimes_{\MZ[\stab(c)]} \MZ^{\chi_c}
\end{equation}
where $\stab(c):=\stab_{\GL(L)}(c)$ is the stabilizer of $c$ in $\GL(L)$ acting on $\MZ$ via the character 
$\chi_c$ induced by the action of $\stab(c)$ on the orientation of $c$. Now let $R_*^c$ be a $\MZ[\stab(c)]$-free 
resolution of $\MZ^{\chi_c}$ then 
\begin{equation}
 \bigoplus_{c \in e_p} \MZ [\GL(L)] \otimes_{\MZ [\stab(c)]} R_*^c
\end{equation}
is a $\MZ [\GL(L)]$ free resolution of $C_p(X)$ and the theorem becomes applicable.

 The algorithm to compute a free resolution in this setup is part of the GAP-package ``HAP'' (\cite{HAP}) and 
needs as input the combinatorial structure of the cell complex and the finite stabilizers.

 The above algorithm cannot only be used for homology computations of the full group $\GL(L)$ but also for
subgroups of finite index in $\GL(L)$ if enough about them is known (e.g. if we can check for membership and 
know the index in $\GL(L)$). For example one can do the computations for unit groups of non-maximal orders or 
special linear groups over imaginary quadratic number fields. Note that if said subgroup is torsion free every stabilizer is trivial and the cellular chain complex is already a free resolution without first applying Wall's perturbation lemma. Furthermore all of the above still holds if we 
quotient by a subgroup acting trivially on $\P$ (in most cases this will only be $\langle -1 \rangle$).

\section{Computational results}\label{Computationalsection}
In this section we want to present some results obtained using the algorithm we described in the previous section. 
\subsection{Groups over imaginary quadratic integers}
We start with an example that is meant as a reliability check for our computations since the homology presented here was already computed by \cite{rahm} and the results presented there match ours in all considered dimensions. 

 Consider the imaginary quadratic number field $K=\MQ(\sqrt{-5})$ and the group 
$\Gamma:=\mathrm{PSL}_2(\MZ_K)$. We obtain a resolution of $\MZ$ over $\MZ[\Gamma]$ 
and compute:
\def \arraystretch{1.3}
{\large
\begin{table}[h]
\begin{center}
 \begin{tabular}{|l|l|}
  \hline
  $n$  & $\H_n \left( \Gamma,\MZ\right)$\\
\hline
\hline
$1$  & $\MZ/2\MZ \times \MZ/6\MZ \times \MZ^2$  \\
\hline 
$2$  & $\MZ/2\MZ \times \MZ/12\MZ \times \MZ$  \\
\hline
$3$  & $(\MZ/2\MZ)^2 \times \MZ/6\MZ$  \\
\hline
$4$  & $(\MZ/2\MZ)^3 \times \MZ/6\MZ$  \\
\hline 
$5$  & $(\MZ/2\MZ)^4 \times \MZ/6\MZ$  \\
\hline
$6$  & $(\MZ/2\MZ)^5 \times \MZ/6\MZ$  \\
\hline
$7$  & $(\MZ/2\MZ)^6 \times \MZ/6\MZ$  \\
\hline 
$8$  & $(\MZ/2\MZ)^7 \times \MZ/6\MZ$ \\
\hline
$9$  & $(\MZ/2\MZ)^8 \times \MZ/6\MZ$  \\
\hline
$10$  & $(\MZ/2\MZ)^9 \times \MZ/6\MZ$  \\
\hline
 \end{tabular}
\end{center}
\caption{The integral homology of $\PSL_2(\MZ_K)$ for $K=\MQ(\sqrt{-5})$}
\end{table}}

If the ring of integers of the imaginary quadratic number field has class number greater than one there are multiple conjugacy classes of maximal orders and the method described here can be used to compute a resolution for each of the unit groups. In many cases this can be used to distinguish between these groups.

Consider the imaginary quadratic number field $\MQ(\sqrt{-6})$. Then its ring of integers $\MZ_K=\MZ\left[\sqrt{-6}\right]$ has class number two and thus there are two isomorphism classes of $\MZ_K$-lattices of dimension two represented by $L_0:=\MZ_K^2$ and $L_1:=\MZ_K \oplus \fp$ where $\fp=\langle 2,\sqrt{-6}\rangle$. Consequently there are two conjugacy classes of maximal orders in $K^{2 \times 2}$ with corresponding unit groups $\GL(L_0)=\GL_2(\MZ_K)$ and $\GL(L_1)$. We compute the integral homology of these two unit groups up to degree $10$ and can for example conclude that $\GL(L_0)$ and $\GL(L_1)$ are not isomorphic as abstract groups.

\def \arraystretch{1.3}
{\large
\begin{table}[h]
\begin{center}
 \begin{tabular}{|l|l|l|}
  \hline
  $n$  & $\H_n \left( \GL(L_0),\MZ\right)$ & $\H_n \left( \GL(L_1),\MZ\right)$\\
\hline
\hline
$1$  & $(\MZ/2\MZ^4$  &$(\MZ/2\MZ)^4$\\
\hline 
$2$  &  $(\MZ/4\MZ)^2 \times \MZ/12\MZ \times \MZ$ & $(\MZ/2\MZ)^2 \times \MZ/12\MZ \times \MZ$\\
\hline
$3$  & $(\MZ/2\MZ)^9 \times \MZ/24\MZ$  &$(\MZ/2\MZ)^8 \times \MZ/24\MZ$\\
\hline
$4$  & $(\MZ/2\MZ)^7$  & $(\MZ/2\MZ)^6 \times \MZ/4\MZ$\\
\hline 
$5$  & $(\MZ/2\MZ)^{13}$ & $(\MZ/2\MZ)^{13}$\\
\hline
$6$  &  $(\MZ/2\MZ)^{8}\times (\MZ/4\MZ)^2 \times \MZ/12\MZ$ & $(\MZ/2\MZ)^{10} \times \MZ/12\MZ$\\
\hline
$7$  &  $(\MZ/2\MZ)^{17} \times \MZ/24\MZ$ & $(\MZ/2\MZ)^{16} \times \MZ/24\MZ$\\
\hline 
$8$  &  $(\MZ/2\MZ)^{15}$& $(\MZ/2\MZ)^{14} \times \MZ/4\MZ$\\
\hline
$9$  &  $(\MZ/2\MZ)^{21}$ &$(\MZ/2\MZ)^{21}$\\
\hline
$10$ &  $(\MZ/2\MZ)^{16} \times  (\MZ/4\MZ)^2\times\MZ/12\MZ$ & $(\MZ/2\MZ)^{18} \times \MZ/12\MZ$\\
\hline
 \end{tabular}
\end{center}
\caption{The integral homology of $\GL_2(\MZ_K)$ and $\GL(\MZ_K\oplus \fp)$ for $K=\MQ(\sqrt{-6})$}
\end{table}}

For linear groups over imaginary quadratic number fields in dimension higher than $2$ only little is known. The only results known to the author appeared in \cite{GunnelsCohomologyOfLinearGroups} where the authors computed the homology in dimension $3$ up to small torsion. Since the complexity of our cell complex grows rapidly with the dimension and apparently with the discriminant of the order (see also \ref{Sizes}) we are not capable of computing the homology in these cases up to a high degree, however we can compute the (so far missing) torsion in low degrees.
  
 We consider the imaginary quadratic number field $K=\MQ(\sqrt{-7})$ and the group $\Gamma:=\GL_3(\MZ_K)$. Using our method we can compute the integral homology up 
to dimension $3$ including torsion. 
\begin{equation}
 \H_n\left( \Gamma, \MZ \right)=\begin{cases}
                                 \MZ/2\MZ & n=1\\
                                 (\MZ/2\MZ)^3 & n=2\\
				  (\MZ/2\MZ)^4 \times (\MZ/4\MZ) \times (\MZ/12\MZ)^2 \times \MZ^2 & n=3
                                \end{cases}
\end{equation}
For $K=\MQ(\sqrt{-1})$ and $\Gamma=\GL_3(\MZ_K)$ we compute the homology up to dimension $2$.
\begin{equation}
 \H_n\left( \Gamma, \MZ \right)=\begin{cases}
                                 \MZ/4\MZ & n=1\\
                                 (\MZ/2\MZ)^2 & n=2
                                \end{cases}
\end{equation}
Note that these results (in particular the dimensions of the free parts) are compatible with those of \cite{GunnelsCohomologyOfLinearGroups}.

\subsection{Groups over real quadratic fields}

Consider the groups $\mathrm{GL}_2(\MZ_K)$ where $K$ is one of the real quadratic fields $\MQ(\sqrt{2})$ and $\MQ(\sqrt{3})$. The groups of this type are closely related to the so-called Hilbert modular groups. We compute the integral homology up to dimension $8$ and $7$, respectively, and obtain:
\def \arraystretch{1.3}
{\large
\begin{table}[h]
\begin{center}
 \begin{tabular}{|l|l|l|}
  \hline
  $n$  & $\H_n \left( \GL_2(\MZ[\sqrt{2}]),\MZ\right)$ & $\H_n \left( \GL_2(\MZ[\sqrt{3}]),\MZ\right)$\\
\hline
\hline
$1$  & $(\MZ/2\MZ)^2 \times \MZ$ & $(\MZ/2\MZ)^2 \times \MZ$\\
\hline 
$2$  & $(\MZ/2\MZ)^6$& $(\MZ/2\MZ)^5 \times \MZ$ \\
\hline
$3$  & $(\MZ/2\MZ)^7 \times \MZ/48\MZ$ & $(\MZ/2\MZ)^6 \times (\MZ/24\MZ)^2 \times \MZ$ \\
\hline
$4$  & $(\MZ/2\MZ)^{11} \times \MZ/24\MZ$ &$(\MZ/2\MZ)^{10} \times \MZ/12\MZ \times \MZ/24\MZ$ \\
\hline 
$5$  & $(\MZ/2\MZ)^{13} \times \MZ/4\MZ$& $(\MZ/2\MZ)^{13} \times \MZ/4\MZ$ \\
\hline
$6$  & $(\MZ/2\MZ)^{18}$ & $(\MZ/2\MZ)^{18}$\\
\hline
$7$  & $(\MZ/2\MZ)^{19} \times \MZ/48\MZ$ & $(\MZ/2\MZ)^{18} \times (\MZ/24\MZ)^2$ \\
\hline
$8$  & $(\MZ/2\MZ)^{23} \times \MZ/24\MZ$ & \\
\hline
 \end{tabular}
\end{center}
\caption{The integral homology of $\GL_2(\MZ_K)$ for $K=\MQ(\sqrt{2}),\MQ(\sqrt{3})$}
\end{table}}

Again we can also consider the case of a number field whose ring of integers is not a principal ideal domain. For $K=\MQ(\sqrt{10})$ there are two isomorphism classes of $\MZ_K$-modules of dimension $2$ represented by $L_0=\MZ_K^2$ and $L_1=\MZ_K \oplus \fp$ where $\fp=\langle 2,\sqrt{10}\rangle$. As was the case for $\MQ(\sqrt{-6})$, the integral homology is sufficient to distinguish between the two groups $\GL(L_0)$ and $\GL(L_1)$.

\def \arraystretch{1.3}
{\large
\begin{table}[h]
\begin{center}
 \begin{tabular}{|l|l|l|}
  \hline
  $n$  & $\H_n \left( \GL(L_0),\MZ\right)$ & $\H_n \left( \GL(L_1),\MZ\right)$\\
\hline
\hline
$1$  &  $(\MZ/2\MZ)^2 \times \MZ$ &$(\MZ/2\MZ)^2 \times \MZ$\\
\hline 
$2$  & $(\MZ/2\MZ)^8 \times \MZ$ & $(\MZ/2\MZ)^7 \times \MZ$\\
\hline
$3$  & $(\MZ/2\MZ)^{12} \times \MZ/12\MZ \times \MZ/24\MZ \times  \MZ$  & $(\MZ/2\MZ)^{11} \times \MZ/6\MZ \times \MZ/24\MZ \times \MZ$\\
\hline
$4$  & $(\MZ/2\MZ)^{19} \times (\MZ/12\MZ)^2$ &$(\MZ/2\MZ)^{18} \times \MZ/6\MZ \times\MZ/12\MZ$ \\
\hline 
$5$  &$(\MZ/2\MZ)^{24} \times (\MZ/4\MZ)^2$ & $(\MZ/2\MZ)^{24} \times \MZ/4\MZ$\\
\hline
$6$  &$(\MZ/2\MZ)^{32} \times \MZ/4\MZ$ & $(\MZ/2\MZ)^{32}$\\
\hline
$7$  & $(\MZ/2\MZ)^{36} \times \MZ/12\MZ \times \MZ/24\MZ $  &$(\MZ/2\MZ)^{35} \times \MZ/6\MZ \times \MZ/24\MZ $\\
\hline 
$8$  & $(\MZ/2\MZ)^{43}  \times (\MZ/12\MZ)^2$ &$(\MZ/2\MZ)^{42} \times \MZ/6\MZ \times \MZ/12\MZ$ \\
\hline

 \end{tabular}
\end{center}
\caption{The integral homology of $\GL_2(\MZ_K)$ and $\GL(\MZ_K\oplus \fp)$ for $K=\MQ(\sqrt{10})$}
\end{table}}

\subsection{A quaternion order}
While we mainly used the described algorithms to work out the homology of Bianchi groups, in principle they work in a very general setup, e.g. for unit groups of maximal orders in quaternion algebras.
 Let $A=\left( \frac{2,3}{\MQ}\right)$ be the rational quaternion algebra ramified at $2$ and $3$ and 
$\mathfrak{M}_{2,3}$ be a maximal order in $A$. Then $\mathfrak{M}_{2,3}^\times$ has periodic integral homology 
(since $\langle -1 \rangle$ is the only torsion) and we obtain for $n \geq 1$:
\begin{equation}
   \mathrm{H}_n(\mathfrak{M}_{2,3}^\times,\mathbb{Z}) = \begin{cases} \MZ/24\MZ & n\equiv 1 \pmod{2} \\ 
\MZ/2\MZ 
& n \equiv 0 \pmod{2}\end{cases}
  \end{equation}

In \cite{Braunetal} one can find an algorithm to compute a presentation of the unit groups we consider. One of the explicit example given is 
\begin{equation}
 \mathfrak{M}_{2,3}^\times/\{\pm 1\} \cong \langle a,b,t~|~a^3,b^2,atbt \rangle.
\end{equation}
 This information can also be used to study the integral homology (cf. \cite{BaumslagIntegralHomology}) and an implementation is available in the software package MAGNUS (cf. \cite{Magnus}). We compared the results obtained from the presentation and from our implementation and are happy to report that both yield the same integral homology, namely:
\begin{equation}
   \mathrm{H}_n(\mathfrak{M}_{2,3}^\times/\{\pm 1\},\mathbb{Z}) = \begin{cases} \MZ/12\MZ & n =1 \\ \MZ/6\MZ & n\equiv 1 \pmod{2},~n>1 \\ 
\{0\} 
& n \equiv 0 \pmod{2}\end{cases}
  \end{equation}

The presentations for the other groups considered in this section appear to be to complicated for MAGNUS to compute the integral homology from them.

\subsection{Additional examples}
For imaginary quadratic number fields $K=\MQ(\sqrt{-d})$ with $1 \leq d \leq 26$ we computed resolutions for the groups $\GL(L)$ and $\SL(L)$ where $L$ runs through a system of representatives of isomorphism classes of lattices in $K^2$ (whence $\End_{\MZ_K}(L)$ runs through a system of representatives of conjugacy classes of maximal orders). These complexes are freely available in the GAP-package HAP (\cite{GAP4,HAP}) and on the author's homepage.

For imaginary quadratic number fields $K=\MQ(\sqrt{-d})$ with $d \in \{1,2,3,7,11,15\}$ we computed contractible complexes for the groups $\GL_3(\MZ_K)$. These complexes are available from the author's homepage in HAP-readable format.

For real quadratic number fields $K=\MQ(\sqrt{d})$ with $2 \leq d \leq 15$ squarefree we computed contractible complexes for the groups $\GL(L)$ where $L$ runs through a system of representatives of the $\MZ_K$-isomorphism classes of modules of dimension $2$. These complexes are also available from the author's homepage in HAP-readable format.

\subsection{Sizes of the cell complexes}\label{Sizes}

As previously noted the complexity of our cell complex appears to increase with growing discriminant. To illustrate this and to give a picture of the scaling of the algorithms we give some information in the following table. We consider the groups of type $\GL_2(\MZ_K)$ and $\GL_3(\MZ_K)$ with $K=\MQ(\sqrt{-D})$ where $-D$ is the discriminant of $\MZ_K$. These groups have virtual cohomological dimension $2$ and $6$, respectively, and we present the number of pairwise inequivalent cells in all relevant dimensions, the maximal number of cells appearing in the boundary of a cell (which is always realized at a cell of maximal dimension), some information on the largest appearing stabilizer, and the time it took to construct the full complex (i.e. without applying Wall's lemma). The computations were performed on an Intel core i7 processor running at $2.93$ GHz.

\def \arraystretch{1.3}
{\large
\begin{table}[h]
\begin{center}
 \begin{tabular}{|l|l|l|l|l|}
  \hline
  $D$  &nbr. of cells &max. bound. & max. stab. & runtime\\
\hline
\hline
\multicolumn{5}{|c|}{$n=2$}  \\
\hline
$3$ &$[ 1, 1, 1 ]$ & $6$ & Order $72$ solvable  &$<1$ sec.\\
\hline 
$4$ & $[ 1, 1, 1 ]$ & $4$ &Order $96$ solvable  & $<1$ sec.\\
\hline
$7$ & $[ 1, 2, 1 ]$ & $6$ & Order $12$ solvable & $<1$ sec.\\
\hline
$8$ & $[ 1, 2, 1 ]$ & $4$ & Order $48$ solvable & $<1$ sec. \\
\hline
$11$ & $[ 1, 2, 1 ]$ & $6$ & Order $24$ solvable & $<1$ sec.\\ 
\hline
$15$ & $[ 2, 4, 4 ]$ & $6$ & Order $12$ solvable  & $<1$ sec.\\ 
\hline
\hline
\multicolumn{5}{|c|}{$n=3$}  \\
\hline
$3$ &$[ 2, 2, 3, 4, 3, 2, 1]$ & $54$ & Order $1296$ solvable &$\sim 4.5$ min.\\
\hline 
$4$ & $[ 1, 1, 3, 5, 4, 3, 2]$ & $76$ & Order $348$ solvable & $\sim 1$ min.\\
\hline
$7$ & $[ 2, 2, 8, 11, 9, 6, 3]$ & $70$ & $C_2 \times \mathrm{PSL}_3(2)$ & $\sim 1$ min.\\
\hline
$8$ & $[ 2, 7, 28, 37, 26, 16, 5]$ & $220$ & Order $96$ solvable & $\sim 6$ min. \\
\hline
$11$ & $[ 12, 51, 125, 150, 91, 34, 8]$ & $478$ & Order $48$ solvable & $\sim 59$ min.\\ 
\hline
$15$ & $[ 11, 86, 299, 454, 338, 137, 28 ]$ & $474$ & Order $48$ solvable & $\sim 16$ h\\ 
\hline
 \end{tabular}
\end{center}
\caption{The well-rounded complexes of $\GL_2$ and $\GL_3$ over imaginary quadratic integers}
\end{table}}

\begin{acknowledgements}
The author is supported by the DFG research training group \emph{Experimental and constructive algebra} (GRK 1632). The author is very thankful to the mathematics department at NUI Galway, and Alexander Rahm in 
particular, for their hospitality
during his stay in Galway in fall 2013. The results presented here are part of the author's Master's thesis which was supervised by Gabriele Nebe.
\end{acknowledgements}

%

\bibliographystyle{spbasic}      
\bibliography{Resolutions_for_unit_groups_of_orders}   

\end{document}